\documentclass[12pt]{amsart}
\usepackage{amsmath,amscd,amsthm,amsfonts, amssymb,amsxtra, mathrsfs,enumerate,relsize}
\usepackage{calrsfs}
\usepackage{caption}
\usepackage[us,12hr]{datetime}
\usepackage[all]{xy} \SelectTips{cm}{}
\usepackage{hyperref}
 \hypersetup{colorlinks=true,citecolor=blue}
   \usepackage{graphicx}
     \usepackage[margin=1.5in]{geometry}
   \usepackage[font=scriptsize]{caption}
\usepackage{todonotes}
\usepackage{tikz}
\usetikzlibrary{shapes.geometric,decorations.pathreplacing}
\usetikzlibrary{positioning}
\CDat
\usepackage{scalerel}[2016/12/29]

\subjclass{Primary: 57P10}
\newtheorem{thm}{Theorem}[section]  
\newtheorem*{un-no-thm}{Theorem}
\newtheorem{prop}[thm]{Proposition}

\newtheorem{bigthm}{Theorem}

\newtheorem{bigadd}[bigthm]{Addendum}

\theoremstyle{definition}
\newtheorem{defn}[thm]{Definition}   

\theoremstyle{definition}
\newtheorem{notation}[thm]{Notation}

\theoremstyle{definition}

\theoremstyle{definition}

\theoremstyle{remark}
\newtheorem{rem}[thm]{Remark}

\newtheorem*{out}{Outline}

\DeclareMathOperator{\Sp}{Sp}

\begin{document}
\title{A relative Euclidean thickening}
\date{\today}
\author{John R.\ Klein}
\address{Wayne State University, Detroit, MI 48202}
\email{klein@math.wayne.edu}
\begin{abstract}  For a finitely dominated CW pair $(K,L)$, we show how to construct
a Poincar\'e triad $(P;\partial_0 P,\partial_1 P)$ and a weak homotopy equivalence
$(K,L) @> {}_{\sim} >> (P,\partial_0 P)$. Furthermore, the Poincar\'e pair $(P,\partial P)$ has a trivial Spivak normal fibration.
 The proof is homotopy-theoretic. We also prove a uniqueness result.
\end{abstract}
\maketitle
\setlength{\parindent}{15pt}
\setlength{\parskip}{1pt plus 0pt minus 1pt}

\def\Sp{\text{\bf Sp}}
\def\vo{\varOmega}
\def\vs{\varSigma}
\def\smsh{\wedge}
\def\flush{\flushpar}
\def\id{\text{id}}
\def\dbslash{/\!\! /}
\def\codim{\text{\rm codim\,}}
\def\:{\colon}
\def\holim{\text{holim\,}}
\def\hocolim{\text{hocolim\,}}
\def\Bbb{\mathbb}
\def\bold{\mathbf}
\def\Aut{\text{\rm Aut}}
\def\cal{\mathcal}
\def\sec{\text{\rm sec}}
\def\gda{G\text{\rm -}\delta\text{\rm -}\alpha}
\def\PDD{\text{\rm pd\,}}
\def\PD{\text{\rm P}}
\def\stableto {\,\, \mapstochar \!\!\to}

\setcounter{tocdepth}{1}
\tableofcontents
\addcontentsline{file}{sec_unit}{entry}
 
 \section{Introduction}
The notion of a manifold thickening of a finite complex was invented by Mazur \cite{Mazur1963} and was further developed by Wall
\cite{Wall_thickening}. The idea is that one can replace the cells of a finite complex $K$
 inductively by handles provided that the geometric dimension of the handles is sufficiently large. 
 The result is a compact smooth manifold $M$ with boundary having the homotopy type of $K$. Then heuristically, $M$
 is a ``regular neighborhood'' of $K$. If the dimension is sufficiently large, then the set of concordance classes
 of manifold thickenings of $K$ is in bijection with $[K,BO]$, the abelian group of isomorphism classes of 
 stable vector bundles over $K$ \cite[prop.~5.1]{Wall_thickening}.

 In a recent manuscript \cite{klein2025poincaresurgery}, it was necessary to deal with the analogous notion of {\it Poincar\'e thickening.} 
 Roughly speaking, the latter consists of a Poincar\'e pair $(P,\partial P)$ together with a choice of homotopy equivalence
 $K \simeq P$.
The collection of Poincar\'e thickenings of a given dimension $d$
is equipped with an equivalence relation, concordance, and the resulting set of equivalence classes is
 denoted by
 \[
 \cal T_d(K)\, .
 \]
 There is also a stabilization map $\sigma \:  \cal T_d(K) \to  \cal T_{d+1}(K)$ which is induced by taking cartesian product with the unit interval. 
 If we let $\cal T_\infty(K)$ denote the colimit
 with respect to stabilization, there is a free and transitive action
 \[
   [K,BG]\times \cal T_\infty(K) \to \cal T_\infty(K)\, ,
  \]
  where  $[K,BG]$ is the abelian group of fiber homotopy types of stable spherical fibrations over $K$ 
  (compare \cite[p.~79]{Mazur1963} in the smooth case).
  The action is defined by twisted suspension with respect to a spherical fibration. 
  
Furthermore, $\cal T_\infty(K)$ is equipped with a preferred basepoint
which is represented by any Poincar\'e thickening
 $(P,\partial P)$ having trivial Spivak normal fibration. 
 We call any such thickening a {\it Euclidean Poincar\'e thickening}. It is unique up to stable concordance.
As the action is free and transitive, the  orbit of the basepoint  defines a bijection
 \[
  \cal T_\infty(K) \cong    [K,BG]\, .
  \]
  In particular, to every stable thickening of $K$ there is a corresponding homotopy class $K \to BG$; call this the {\it classifying map} of the thickening.
 If one is willing to appeal to manifolds, then 
 Spivak's paper \cite{Spivak} provides a construction of a Euclidean thickening when $K$ is a finite complex.
 In \cite{Klein_dualizing} we gave an alternative homotopy-theoretic construction.
  The goal of this paper is to {\it relativize} the result.

  \subsection{Relative Poincar\'e thickenings}
 
 Let $(P;\partial_0 P,\partial_1 P)$ be a Poincar\'e triad of dimension $d$. 
 We set \[
 \partial_{01} P = \partial_0 P \cap \partial_1 P\, .
 \]
 Then $(\partial_0 P,\partial_{01} P)$ and
 $(\partial_1 P,\partial_{01} P)$ are Poincar\'e pairs of dimension $d-1$, and $\partial P =\partial_0 P \cup_{\partial_{01} P} \partial_1 P$
 is such that $(P,\partial P)$ is a Poincar\'e pair of dimension $d$.

 \begin{defn}[{\cite[p.~764]{GK}}] \label{defn:spine} The {\it homotopy spine dimension} of $(P;\partial_0 P,\partial_1 P)$
  (relative to $\partial_0 P$) is $\le p$ 
if
\begin{enumerate}[(i).]
\item the pair $(P,\partial_0 P)$ is homotopically $p$-dimensional (i.e., up to homotopy, $(P,\partial_0 P)$ is a retract of a CW pair of relative dimension $\le p$), and
\item the pair $(P,\partial_1 P)$ is $(d-p-1)$-connected.
\end{enumerate}
\end{defn}

 \begin{defn} \label{defn:rel-thickening}
Let $(K,L)$ be a cofibrant  pair.
A {\it Poincar\'e thickening} of $(K,L)$ of dimension $d$ 
is a pair
\[
((P;\partial_0 P,\partial_1 P), f)
\]
consisting of a Poincar\'e triad 
\[
(P;\partial_0 P,\partial_1 P)
\]
of dimension $d$ having  homotopy spine dimension $\le d-3$, 
and a weak homotopy equivalence 
\[
f\: (K,L) @> \sim >> (P,\partial_0 P)\, .
\]  
In addition, we require that
the pair $((\partial_0 P,\partial_{01} P),f_{|L})$ is a Poincar\'e thickening of dimension $d-1$.
\end{defn}

\begin{rem} The definition 
codifies the notion of  ``regular neighborhood''  $P$ of $K$ up to homotopy 
which meets the boundary $\partial P$ in a ``regular neighborhood'' $\partial_0 P$ of $L$.
\end{rem}

\begin{notation}  To avoid clutter, we set 
\[
(P,f) := ((P;\partial_0 P,\partial_1 P), f)\, ,
\]
whenever the triad structure on $P$ is understood.
\end{notation}

Poincar\'e thickenings $(P, f)$ and $(Q,g)$ of $(K,L)$ 
are {\it concordant} if there is a weak homotopy equivalence of triads 
\[
h\: (P;\partial_0 P,\partial_1 P) @>\sim >> (Q;\partial_0 Q,\partial_1 Q)
\]
such that $h\circ f$ is homotopic to $g$.

\begin{defn} 
We say that $(P,f)$ is {\it Euclidean} if the Spivak normal  fibration 
of $(P,\partial P)$ is trivializable. 
\end{defn}

The main result is then the following:

\begin{bigthm} \label{bigthm:relative} There exists a Euclidean Poincar\'e  thickening
of $(K,L)$.
\end{bigthm}

\begin{rem} A statement like Theorem \ref{bigthm:relative} is needed in \cite{Klein_compression2}. The latter paper
is a component of the author's approach to non-simply connected Poincar\'e surgery.

A relatively easy proof of the result using manifold transversality exists (see  \cite[\S3]{Klein_haef2} for a related result).  The current paper provides an alternative homotopy-theoretic approach with the intention of keeping the results of Poincar\'e surgery manifold-free.
\end{rem}

The suspension of a Poincar\'e triad $(P;\partial_0 P,\partial_1 P)$ of dimension $d$ is the Poincar\'e triad of dimension $d+1$  
\[
(P \times I, (\partial_0 P) \times I, S_P \partial_ 1P)\, ,
\]
where $S_P \partial_ 1P$ is the unreduced fiberwise suspension of $\partial_1 P$ over $P$.
Iterated suspension induces an operation on Poincar\'e thickenings of $(K,L)$. Two Poincar\'e thickenings of $(K,L)$
are {\it stably concordant} if they are concordant after some iterated suspension. 

\begin{bigthm} \label{bigthm:unique} There is only one Euclidean Poincar\'e  thickening
of $(K,L)$  up to stable concordance.
\end{bigthm}

The following addendum relativizes Theorem \ref{bigthm:unique}. For a more detailed statement, see Proposition \ref{prop:unique}  below.

\begin{bigadd} \label{bigadd:unique}
Suppose that two Euclidean relative thickenings of $(K,L)$ coincide as thickenings along $L$. If trivializations
of their Spivak normal fibrations can be chosen to agree on the restriction to $L$, then the 
thickenings of $(K,L)$ are stably concordant relative to the given thickening of $L$.
\end{bigadd}

 \subsection{Conventions} We work in the convenient model category of
compactly generated weak Hausdorff spaces \cite[thm.~2.4.25]{Hovey}.
Poincar\'e duality spaces and pairs are assumed to be finitely dominated.

\begin{out} In Section \ref{sec:thickenings} we recall the notion of Poincar\'e thickening and we prove that
any map from a finitely dominated CW complex to a Poincar\'e space admits an embedded thickening after taking the
cartesian product with a disk of sufficiently high dimension. Section \ref{sec:rel-thickening} contains the proof of 
Theorems \ref{bigthm:relative} and \ref{bigthm:unique}. Section \ref{sec:rel-unique} contains a proof of Addendum 
\ref{bigadd:unique}.
 \end{out}
 
 \section{Embedded Poincar\'e thickenings} \label{sec:thickenings}

Let $K$ be a finitely dominated CW complex of dimension $\le k$.
 
\begin{defn}   A {\it Poincar\'e thickening} of $K$ of dimension $d$ consists of
 \[
 ((P,\partial P), h)
 \]
 in which
 \begin{enumerate}[(i).]
 \item  $(P,\partial P)$ is a Poincar\'e pair which is $(d-k-1)$-connected, and
 \item $h\: K \to P$ is a weak homotopy equivalence.
 \end{enumerate}
 \end{defn}
 
 In general, we assume that $k \le d-3$ so that $(P,\partial P)$ is $2$-connected.
 When the boundary $\partial P$ is understood, we will use the notation $(P,h)$ for the Poincar\'e thickening.
 
If $(P,h)$ is a Poincar\'e thickening of dimension $d$, then 
\[
\sigma^j(P,h) := (P\times D^j,h_j)
\]
 is a Poincar\'e thickening of dimension $d+j$, where
\[
h_j\: K @> h >> P \subset P \times D^j\, .
\]
A {\it concordance} of $d$-dimensional Poincar\'e thickenings  from $(P, h_0)$ to $ (Q, h_1)$ 
is a weak homotopy equivalence $H\: (P,\partial P) @> \simeq >> (Q,\partial Q)$ such that
$H \circ h_0$ is homotopic to $h_1$.

A Poincar\'e thickening $(P,h)$  of $K$ is  {\it Euclidean} if $(P,\partial P)$ has trivial Spivak normal fibration.

\begin{prop} \label{prop:stable-euclid} There exists a Euclidean thickening of $K$ which is unique up to stable concordance.
\end{prop}

\begin{rem} In the finite CW case, the result follows from \cite{Spivak, Wall_thickening} if one is willing to appeal to manifold theory.
A homotopy-theoretic proof can be found  in \cite{Klein_dualizing}. See also \cite[thm.~B]{Klein_immersion}.
\end{rem}

Fix a Poincar\'e pair  $(M,\partial M)$ of dimension $d$. Let
\[
\cal I(\partial M)
\]
be the category whose objects are Poincar\'e pairs $(N,\partial N)$ with $\partial N = \partial M$.
A morphism $(N,\partial N) \to (N',\partial N')$ is a weak homotopy equivalence of pairs which
restricts to the identity on $\partial N$.   When the boundary is understood, we abbreviate notation 
and write $N$ in place of $(N,\partial N)$.

Given a Poincar\'e pair $(P,\partial P)$ of dimension $d$, we have a functor
\[
\kappa_P\: \cal I(\partial P \amalg \partial M) \to \cal I(\partial M)
\]
defined by $C\mapsto P \cup_{\partial P} C$.

\begin{defn}[{\cite{GK,Klein_compression2}}] A {\it Poincar\'e embedding} of $P$ in $M$ is an object of the
over category
\[
\kappa_P/M\, .
\]
\end{defn}

\begin{rem} An unraveling of the definition shows that an object of $\kappa_P/M$ amounts to specifying data of the form
 \[
 ((N;P,C),h)
 \]
in which 
\begin{enumerate}[(i).]
\item $(N;P,C)$ is a Poincar\'e triad in which $\partial P = P \cap C$ and  $\partial N = \partial M$,
\item $\partial C = \partial P \amalg \partial N$, and
 \item  $h\: (N,\partial N) @>{}_\sim >> (M,\partial M)$ is a weak equivalence of pairs
 which restricts to the identity on $\partial M$.
 \end{enumerate}
We call $h_{|P}\: P \to M$ the {\it underlying map} of the Poincar\'e embedding.
 \end{rem}

 \begin{rem} A Poincar\'e embedding of $P$ in $M$ with associated triad $(N;P,C)$ 
 determines a Poincar\'e embedding of $P \times D^j$ in $M \times D^j$ with triad
 \[
 (N\times D^j; P \times D^j,S^j_N C)\, ,
 \] 
where 
\[
S^j_N C = C\times D^j \cup_{C\times S^{j-1}} (N \times S^{j-1})
\]
is the $j$-fold unreduced fiberwise suspension of $C$ over $N$.
 \end{rem}

\begin{defn}[{\cite{Klein_haef}}] Let $f\: K \to M$ be a map, where $(M,\partial M)$ is a Poincar\'e pair of dimension $d$. 
An {\it embedded thickening} of $f$ consists of a Poincar\'e thickening $(P,h)$ of $K$ and a codimension-zero
Poincar\'e embedding with underlying map $F\:P \to M$ such that the composition
\[
K @> h >> P @> F >> M
\]
is homotopic to $f$.
\end{defn}
In the setting of the definition, we say that $F\circ h\: K \to M$ underlies an embedded thickening.

\begin{rem} An alternative way to describe an embedded thickening is to display a diagram
\[
\xymatrix{
& \partial P \ar[r]^{\subset} \ar[d]_{\cap} & C \ar[d] & \partial M \ar[l]\ar[dl] \\
K \ar[r]_h^{\simeq} & P \ar[r]_F & M
}
\]
in which the displayed square is a homotopy pushout.
Since $h$ is a weak homotopy equivalence, up to homotopy, the data amount to specifying
a commutative diagram of the form
\[
\xymatrix{
A \ar[r] \ar[d] & C \ar[d] & \partial M \ar[l]\ar[dl] \\
K \ar[r] & M
}
\]
in which the square is a homotopy pushout, where if $\bar K$is the mapping cylinder of the map $A\to K$, 
then $(\bar K,A) \simeq (P,\partial P)$.
\end{rem}

Let $f\: K \to M$ be a map, where $(M,\partial M)$ is a Poincar\'e pair of dimension $d$. 

\begin{thm}\label{thm:embedded-thickening} There is
an integer $j > 0$, such that the composition
\[
K @> f >> M @> \subset >> M \times D^j
\]
underlies an embedded thickening.
\end{thm}

\begin{proof}The proof is essentially  that of \cite[thm.~A]{Klein_dualizing2}.
In view of the fact that $\cal T_\infty(K)$ is a $[K,BG]$-torsor, it will be enough to show that there is a Euclidean thickening
of $M$, say $(\bar M,h)$,
together with an embedded Poincar\'e thickening of the composition
\[
K @> f >> M  @> h >> \bar M\, .
\]
If  $X$ is a space, we let $\Sp_{/X}$ be one of the model categories of parametrized spectra over $X$ 
\cite{Malkiewich_convenient, Hebestreit-et-al, May-Sigurdsson}.

In \cite[\S5]{Klein_dualizing2} we defined a parametrized spectrum
\[
\cal D_X \to X\, 
\]
whose fiber at $x\in X$ is a model for the dualizing spectrum $D_{\Omega_x X}$ of \cite{Klein_dualizing}, where  $\Omega_x X$ is a topological group model for the based
loop space of $X$.

A more modern formulation is to describe $\cal D_X$ as
\[
R{p_2}_\ast X^+\, ,
\]
where $X^+ = (X \times X) \amalg X$ (which after taking its fiberwise suspension spectrum, 
we may consider as an object of the category of parametrized spectra over $X\times X$),  
the map $p_2\: X \times X \to X$ is the
second factor projection, and
\[
R{p_2}_\ast \: \text{ho\,} \Sp_{/X\times X} \to \text{ho\,}\Sp_{/X}
\]
denotes the right derived functor of the pushforward  \cite[p.~42]{Malkiewich_convenient}, \cite[prop.~71]{Hebestreit-et-al}.
(For a map of spaces $f\: X\to Y$ the
pushforward $f_{\ast}$ is the right adjoint  to the pullback functor $f^\ast\: \Sp_{/Y} \to \Sp_{/X}$, the latter assigning to a parametrized spectrum over $Y$ its pullback to $X$.)

Similarly, if $a\: A\to X$ is a map, one has a {\it relative dualizing spectrum}
\[
\cal D_{A\to X} \to X
\]
which is given by the right derived functor $R{q_2}_\ast$, where $q_2$ is the composition
\[
A\times X \to X\times X @>p_2 >> X\, .
\]
Then one has a commutative diagram
\[
\xymatrix{
\cal D_{K} \ar[r] \ar[d] & \cal D_{K \to M} \ar[d]& \cal D_M \ar[dl]  \ar[l] \\
K \ar[r] _f & M 
}
\]
in which downward arrows are the structure maps of parametrized spectra (cf.~\cite[\S6]{Klein_dualizing2}).

As in \cite[cor.~51]{Klein_dualizing}, there is an integer $j \gg d$ such that the $j$-fold
fiberwise suspension applied to the total spaces of the parametrized spectra in the diagram yields
a diagram of spaces
\[
\xymatrix{
E_K \ar[r] \ar[d] &E_{K \to M} \ar[d]& E_M \ar[dl]  \ar[l]  \\
K \ar[r] & M & 
}
\]
where the downward pointing maps are fibrations. Set $\partial \bar M =E_M$
and replace the map 
\[
 \partial \bar M \to M
\]
by a cofibration $\partial \bar M \to \bar M$.  Then $(\bar M,\partial \bar M)$
is a Poincar\'e pair of dimension $j-d$ with trivial Spivak normal fibration, and the diagram
\[
\xymatrix{
E_K \ar[r] \ar[d] &E_{K \to M} \ar[d]& \partial \bar M \ar[l] \ar[dl] \\
K \ar[r] & \bar M
}
\]
defines an embedded Poincar\'e thickening of the map $K \to \bar M$.
\end{proof}

\section{Proof of Theorems \ref{bigthm:relative} and \ref{bigthm:unique}} \label{sec:rel-thickening}

 \begin{proof}[Proof of Theorem \ref{bigthm:relative}] Let $(K,L)$ be a finitely dominated CW pair. Choose a Euclidean thickening of $K$,
 say, $(P,f)$, where $f\: K \to P$ is a weak homotopy equivalence. Note that $f$  determines a map of pairs
 $(K,L) \to (P \times D^1,\partial (P\times D^1))$. For this reason, we may as well assume at the outset
$f$ is a map of pairs $(K,L) \to (P,\partial P)$. 

If $j$ is sufficiently large, then by Theorem \ref{thm:embedded-thickening}, the map 
\[
L \to \partial P @> \subset >> (\partial P) \times D^j
\]
underlies an embedded thickening of $L$ in $ (\partial P) \times D^j$, say 
\[
\xymatrix{
\partial Q \ar[r] \ar[d] & C \ar[d] & \partial P \times S^{j-1} \ar[l] \ar[dl]\\
Q \ar[r] &  \partial P \times D^j
}
\]
where $(Q,\partial Q)$ thickens $L$. Moreover, $(Q,\partial Q)$ has trivial Spivak normal  fibration.

Then we have a weak homotopy equivalence
\[
Q \cup_{\partial Q} (C \cup_{\partial P \times S^{j-1}} P \times S^{j-1}) \simeq \partial (P \times D^j)\, .
 \]
 Without loss of generality, we take this to be an identification.
Setting $P_j = P \times D^j$, we have a Poincar\'e triad
\[
(P_j; Q,W),
\]
where  $W = C \cup_{\partial P \times S^{j-1}} P \times S^{j-1}$, which relatively thickens the pair $(K,L)$, and has trivial Spivak normal fibration.
\end{proof}

\begin{proof}[Proof of Theorem \ref{bigthm:unique}]
Let
\[
 ((P;\partial_0P,\partial_1P),f)
 \qquad\text{and}\qquad
 ((Q;\partial_0Q,\partial_1Q),g)
\]
be Euclidean Poincar\'e thickenings of $(K,L)$. On forgetting the
decomposition of the boundary, these determine Euclidean Poincar\'e
thickenings of $K$. After stabilization they are concordant by
Proposition~2.2. We may therefore assume that
\[
 (P,\partial P)=(Q,\partial Q)
\]
and that there is a homotopy
\[
 F\colon K\times I\longrightarrow P
\]
from $f$ to $g$.

After one further stabilization, we may assume that $F$ factors through
$\partial P$. Indeed, replace $P$ by $P\times I$ and compose $F$ with
the inclusion
\[
 P\times 0\subset \partial(P\times I).
\]
Thus
\[
 F|_{L\times I}\colon L\times I\longrightarrow \partial P
\]
is a homotopy from $f|_L$ to $g|_L$.

The pairs
\[
 (\partial_0P,\partial_{01}P)
 \qquad\text{and}\qquad
 (\partial_0Q,\partial_{01}Q)
\]
determine embedded Poincar\'e thickenings of the endpoint maps
\[
 f|_L,g|_L\colon L\longrightarrow \partial P.
\]
After further stabilization, these are concordant along
$F|_{L\times I}$ by \cite[thm.~A]{Klein_haef}. Set
\[
 \partial P= N =\partial Q.
\]
After $j$ decompressions, the complements of the two embedded
thickenings are
\[
 C_P=S_N^j\partial_1P
 \qquad\text{and}\qquad
 C_Q=S_N^j\partial_1Q.
\]
By the definition of concordance of embedded thickenings, $C_P$ and
$C_Q$ are weakly equivalent relative to the common subspace
$N\times S^{j-1}$. Likewise, the source thickenings
\[
 \partial_0P\times D^j
 \qquad\text{and}\qquad
 \partial_0Q\times D^j
\]
are weakly equivalent.

On the other hand,
\[
 S_P^j\partial_1P
 =
 C_P\cup_{N\times S^{j-1}}(P\times S^{j-1})
\]
and
\[
 S_P^j\partial_1Q
 =
 C_Q\cup_{N\times S^{j-1}}(P\times S^{j-1}).
\]
Consequently, gluing the equivalence $C_P\simeq C_Q$ to the identity
of $P\times S^{j-1}$ gives a weak equivalence
\[
 S_P^j\partial_1P
 \ \simeq\
 S_P^j\partial_1Q.
\]
Together with the weak equivalence of the $\partial_0$-faces and the
identity of $P\times D^j$, this gives a weak equivalence of Poincar\'e
triads
\[
 \bigl(P\times D^j;\partial_0P\times D^j,
       S_P^j\partial_1P\bigr)
 \ \simeq\
 \bigl(P\times D^j;\partial_0Q\times D^j,
       S_P^j\partial_1Q\bigr).
\]
The equivalence is compatible, up to homotopy, with the maps from
$(K,L)$, by the choice of $F$. Hence the original relative
Poincar\'e thickenings are stably concordant.
\end{proof}

  \section{Proof of Addendum \ref{bigadd:unique}}  \label{sec:rel-unique}
 
 Addendum  \ref{bigadd:unique} may be more precisely formulated as the following result.
 
 \begin{prop} \label{prop:unique}
Let $(K,L)$ be a finitely dominated CW pair with
$L\ne\emptyset$. Let
\[
((P;A,\partial_1 P),f) \qquad\text{and}\qquad ((Q;A,\partial_1 Q),g)
\]
be relative Euclidean  thickenings of $(K,L)$ where
$f:(K,L)\xrightarrow{\simeq}(P,A)$ and
$g:(K,L)\xrightarrow{\simeq}(Q,A)$ are maps of pairs  whose restrictions to $L$ coincide.

Write $D=\partial A=A\cap \partial_1 P=A\cap \partial_1 Q$.
Suppose that stable trivializations of the Spivak normal fibrations of $P$ and $Q$
can be chosen whose induced trivializations on $(A,D)$ agree up to
homotopy, using the boundary identifications.
Then, after stabilization, there is a weak equivalence of triads
\[
h:(P;A,\partial_1 P)\xrightarrow{\simeq}(Q;A,\partial_1 Q)
\]
such that $h|_A=\operatorname{id}_A$ and
$hf\simeq g\ \text{rel }L$.
\end{prop}

\begin{proof} We work stably.
Set
\[
\begin{aligned}
Y&=K\times I,
& X&=(K\times\partial I)\cup(L\times I),\\
R&=P\cup_A(A\times I)\cup_A Q,
& \partial R&=\partial_1 P\cup_D(D\times I)\cup_D \partial_1 Q.
\end{aligned}
\]
The maps $f$ and $g$ and their common restriction on $L$ determine a weak equivalence $j\:X\xrightarrow{\sim}R$.
Moreover,  by the compatibility hypothesis, the trivializations of the Spivak fibrations $P$ and $Q$ 
glue across $A\times I$, and yield a trivialization of  the Spivak fibration of $(R,\partial R)$.
Consequently, 
$(R,\partial R,j)$ is a Euclidean thickening of $X$.

Apply Theorem \ref{bigthm:relative} to $(Y,X)$, obtaining a relative Euclidean thickening
$(T_0;R_0,C_0)$. By 
Proposition \ref{prop:stable-euclid}, the thickenings defined by $(R_0,\partial R_0)$ and $(R,\partial R)$ of $X$ are
stably concordant. Attaching the mapping cylinder of such a concordance
replaces the $R_0$-face by $R$ and gives a Poincar\'e triad
$(T;R,C)$. By the homotopy extension property a weak equivalence
\[
F:(Y,X)\xrightarrow{\sim}(T,R)
\]
can be chosen to restrict to $j$.
In particular, $P\to T$ and $Q\to T$ are weak equivalences, and
$F(l,t)=(f(l),t)$ on $L\times I$.
Furthermore, a Poincar\'e-Lefschetz duality argument 
which we omit shows that $\partial_1 P\to C$ and $\partial_1 Q\to C$ are weak equivalences.

Finally, put 
\[
U = Q \cup_{A\times \{1\}} A\times I \quad \text{ and } \quad U_C=U\cap C=\partial_1 Q\cup_D(D\times I)\,  .
\]
The identity on $Q$ and the cylinder projections define a map
$(U,U_C)\to(Q,\partial_1 Q)$. Then the above equivalences imply that the maps
\[
U_C\hookrightarrow C
\qquad\text{and}\qquad
U\cup_{U_C}C\hookrightarrow T
\]
are acyclic cofibrations. Extend first over $C$ and then
over $T$ to obtain a map of pairs
$r:(T,C)\to(Q,\partial_1 Q)$ which is the identity on $Q$ and is projection
on $A\times I$. Then $h=r|_P$ is a weak equivalence of triads
fixing $A$: this follows from $P\simeq T\simeq Q$ and
$\partial_1 P\simeq C\simeq \partial_1 Q$.
The homotopy $r\circ F$ has endpoints $h\circ f$ and $g$ and is constant on
$L\times I$. Hence, $h\circ f\simeq g\ \text{rel }L$, as required.
\end{proof}


 \end{document}